\documentclass[11pt]{article}

\usepackage{amsmath,amssymb,amsthm}
\usepackage{mathtools}
\usepackage{geometry}
\newtheorem{theorem}{Theorem}[section]
\newtheorem{lemma}[theorem]{Lemma}
\newtheorem{proposition}[theorem]{Proposition}
\newtheorem{remark}[theorem]{Remark}
\newtheorem{definition}[theorem]{Definition}

\title{A Recursive Representation of Compatibility Matrices\\
for Transported Differential Operators}

\author{Gerardo Hern\'andez-del-Valle
\thanks{Centro de Estudios Monetarios Latinoamericanos (CEMLA), Mexico City, Mexico.
E-mail: \texttt{ghernandez@cemla.org}}
}

\date{}

\begin{document}

\maketitle

\begin{abstract}

We develop a recursive algebraic framework for the compatibility matrices arising in moving-boundary problems for the heat equation. Such matrices naturally appear in the analysis of boundary trace operators, including Dirichlet-to-Neumann maps and first-passage problems for Brownian motion. The
construction is based on two recursive hierarchies: transported
differential operators in the interior and boundary operators generated by
successive differentiation of the boundary condition. Their interaction
produces a recursive family of compatibility matrices whose size grows
linearly with the order of the transported operator.

The main contribution of the paper is a decomposition of every
compatibility matrix into a universal reference matrix and a sequence of
transport perturbations. This representation isolates the boundary
geometry from the recursive transport corrections and shows that only the
transport rows change at each recursive level. Consequently, the
determinant of the compatibility matrix may be interpreted as the
determinant of a reference matrix together with successive low-rank
perturbations.

The recursive construction is illustrated through explicit examples, which
suggest additional algebraic structure in the reference determinants. The resulting recursive representation provides a systematic organization of compatibility matrices and establishes an algebraic foundation for future investigations of recursive determinant formulas, boundary trace operators, and Dirichlet-to-Neumann maps for moving-boundary problems.

\end{abstract}

\noindent\textbf{Keywords:}
Compatibility matrices; transported differential operators; heat equation;
moving boundaries; recursive operator methods; low-rank perturbations.\\
\noindent\textbf{2020 Mathematics Subject Classification.}
Primary 35R37; Secondary 35K05, 15A15.
\section{Introduction}

Moving-boundary problems for the heat equation arise in a variety of
mathematical settings, including free-boundary problems, first-passage
distributions of Brownian motion, and the construction of
Dirichlet-to-Neumann maps. In each of these settings one is naturally led
to study differential operators evaluated along a moving boundary,
together with the compatibility conditions that couple the interior
evolution to the boundary data (see, for example,
\cite{hernandez,HerGuerra,fokas}).

For low-order differential operators these compatibility conditions can be
computed explicitly. However, the algebra rapidly becomes intractable as
the order of the operator increases. Already for third-order operators,
the associated compatibility matrix involves a significant number of
coefficients together with nonlinear interactions between the transported
operator and the boundary hierarchy. Consequently, direct computations
quickly become cumbersome, suggesting that a more structural approach is
required.

The purpose of this work is to show that the transported differential
operators possess a natural recursive structure and that this recursion is
inherited by the associated compatibility matrices. More precisely, we
consider a hierarchy of differential operators
\[
L_0,L_1,\ldots,L_n,
\]
where each operator is obtained from the previous one by adding a single
higher-order layer. After transport by the heat semigroup, this hierarchy
produces transported operators
\[
N_0,N_1,\ldots,N_n,
\]
which satisfy an explicit recursive representation in terms of heat
polynomials and a sequence of time-dependent coefficient functions. This
construction identifies the transported operator of order $n$ as a
recursive perturbation of the transported operator of order $n-1$.

The principal observation of the paper is that this recursive description
of the transported operators induces a corresponding recursive
representation of the compatibility matrices. If $M_n$ denotes the
compatibility matrix associated with the transported operator $N_n$, then
$M_n$ admits a decomposition into a universal reference matrix together
with a sequence of transport perturbations. The reference matrix depends
only on the recursive boundary hierarchy, while each perturbation modifies
only the transport equations, leaving the boundary operators unchanged.

As a consequence, every recursive step introduces a perturbation whose
rank is considerably smaller than the dimension of the compatibility
matrix. From this perspective, the compatibility matrices may be viewed as
an algebraic representation of the hierarchy of boundary traces generated
by the transported operators. One immediate consequence of this
decomposition is that determinant identities may be derived recursively
through successive applications of the Matrix Determinant Lemma.

Although the explicit computations presented here are carried out only for
second- and third-order operators, the recursive construction naturally
extends to arbitrary order. Rather than treating each compatibility matrix
as an isolated algebraic object, the present approach places them within a
recursive hierarchy whose structure is inherited directly from the
transported differential operators. The resulting framework emphasizes the
underlying algebraic organization of compatibility matrices independently
of any particular moving-boundary problem and provides a natural starting
point for future investigations of recursive determinant formulas,
boundary trace operators, and Dirichlet-to-Neumann maps.

The paper is organized as follows. Section~\ref{Sec_2} introduces the
family of differential operators and defines their transported
counterparts. Section~\ref{Sec_3} establishes the recursive transport formula for
the transported operators. Section~\ref{Sec_4} develops the hierarchy of boundary
operators associated with moving boundaries. Section~\ref{Sec_5} introduces the
compatibility matrices and explains how they encode the interaction
between the transported operators and the boundary hierarchy.
Section~\ref{Sec_6} proves the recursive decomposition of the compatibility
matrices into a universal reference matrix together with successive
transport perturbations. Section~\ref{Sec_7} shows how this decomposition leads
naturally to a sequence of low-rank perturbations and recursive
determinant representations via the Matrix Determinant Lemma.
Section~\ref{Sec_8} presents explicit second- and third-order examples that
illustrate the recursive construction. Finally, Section~\ref{Sec_9} concludes with
a discussion of the recursive framework and several directions for
future research.
\section{The family of differential operators}\label{Sec_2}

Throughout this work, $D$ denotes differentiation with respect to the
spatial variable,
\[
D=\frac{d}{dx}.
\]

We begin by introducing the class of differential operators that will be
studied throughout the paper.

\begin{definition}\label{Def_Operator_L}
For each integer $n\ge0$, let $\mathcal L_n$ denote the collection of
linear differential operators of order $n$ of the form
\begin{equation}\label{operator_L}
L_n
=
\sum_{j=0}^{n}
p_j^{(n)}(x)D^j,
\end{equation}
where the coefficient multiplying $D^j$ is a polynomial of degree at most
$n-j$,
\begin{equation}\label{coeff_poly}
p_j^{(n)}(x)
=
\sum_{k=0}^{\,n-j}
c_{j,k}x^k,
\qquad
0\le j\le n.
\end{equation}
The constants
$c_{j,k}$
are assumed to be independent of $x$.
\end{definition}

The first members of this family are
\[
L_0
=
c_{0,0},
\]
\[
L_1
=
c_{1,0}D
+
\left(
c_{0,0}
+
c_{0,1}x
\right),
\]
\[
L_2
=
c_{2,0}D^2
+
\left(
c_{1,0}
+
c_{1,1}x
\right)D
+
\left(
c_{0,0}
+
c_{0,1}x
+
c_{0,2}x^2
\right),
\]
and
\[
\begin{aligned}
L_3
={}&
c_{3,0}D^3
+
\left(
c_{2,0}
+
c_{2,1}x
\right)D^2
\\
&
+
\left(
c_{1,0}
+
c_{1,1}x
+
c_{1,2}x^2
\right)D
\\
&
+
\left(
c_{0,0}
+
c_{0,1}x
+
c_{0,2}x^2
+
c_{0,3}x^3
\right).
\end{aligned}
\]

\begin{remark}
The indexing has been chosen so that the first index records the order of
the derivative, while the second index corresponds to the degree of the
associated monomial in $x$. Consequently, the coefficient multiplying
$D^j$ is always a polynomial of degree at most $n-j$. This convention
provides a uniform description of the entire family of operators and will
prove particularly convenient when studying their recursive transport.
\end{remark}

The family $\{L_n\}_{n\ge0}$ is naturally nested: passing from $L_{n-1}$ to
$L_n$ introduces one additional derivative together with one additional
coefficient in each of the lower-order polynomial coefficients. This
recursive structure is the starting point for the transport construction
developed in the following section.

\section{Recursive transport}\label{Sec_3}

The objective of this section is to transport the family of differential
operators introduced in the previous section through the heat semigroup.
The resulting transported operators preserve the polynomial structure of
the original family and admit a remarkably simple recursive description.


\begin{definition}
Let
\[
P_t
=
\exp\!\left(\frac{t}{2}D^2\right)
\]
denote the one-dimensional heat semigroup.

For every differential operator
\[
L_n\in\mathcal L_n
\]
as in Definition~\ref{Def_Operator_L},
we define its
\emph{transported operator} by

\[
N_n
=
P_tL_nP_{-t}.
\]

The notation
\[
\gamma_f(N_n)
\]
denotes the restriction of the coefficients of $N_n$ to the moving
boundary
\(
x=f(t).
\)
\end{definition}


The heat semigroup preserves polynomial differential operators. More
precisely, the transported operators remain of order $n$, although their
coefficients now depend on both the spatial variable $x$ and the time
variable $t$.

The following theorem is the fundamental recursive identity underlying the
entire construction.


\begin{theorem}[Recursive transport formula]
For every integer $n\ge1$, the transported operators satisfy
\begin{equation}\label{operator_N}
N_n
=
N_{n-1}
+
\sum_{j=0}^{n}
\frac{H_j(x,t)}{j!}
g_n^{(j)}(t)
D^{\,n-j},
\end{equation}
where $H_j(x,t)$ denotes the $j$-th classical heat polynomial introduced by
Rosenbloom and Widder \cite{Widder}, and
\[
g_n(t)
=
\sum_{k=0}^{n}
c_{0,k}t^k
\]
is obtained by replacing the spatial variable in the zeroth-order
polynomial coefficient of $L_n$ by the time variable.
\end{theorem}


\begin{proof}

The proof proceeds by induction on $n$.

For $n=1$ the identity follows directly from the transport formulas for
polynomials under the heat semigroup.

Assume that the formula holds for $N_{n-1}$.

The operator $L_n$ differs from $L_{n-1}$ only through the additional
polynomial coefficient associated with the identity operator. Conjugating
this polynomial by the heat semigroup produces its classical heat-polynomial
expansion,

\[
P_tg_n(x)P_{-t}
=
\sum_{j=0}^{n}
\frac{H_j(x,t)}{j!}
g_n^{(j)}(t).
\]

Since multiplication by $D^{\,n-j}$ commutes with the transport, the new
terms are appended to the transported operator of the previous level,
yielding

\[
N_n
=
N_{n-1}
+
\sum_{j=0}^{n}
\frac{H_j}{j!}
g_n^{(j)}
D^{\,n-j},
\]
which completes the induction.
\end{proof}


Using the recursive transport formula \eqref{operator_N} together with the
representation \eqref{operator_L} of the operators in $\mathcal L_n$, the
first transported operators are
\[
N_1
=
g_1(t)D
+
(c_{0,1}x+c_{0,0}),
\]
\[
\begin{aligned}
N_2
={}&
g_2(t)D^2
\\
&
+
\left(
g_1(t)
+
(c_{1,1}+2c_{0,2}t)x
\right)D
\\
&
+
\left(
c_{0,0}+c_{0,1}x+c_{0,2}(x^2+t)
\right),
\end{aligned}
\]
while the expression for $N_3$ already illustrates the recursive nature
of the construction,
\[
N_3
=
N_2
+
\sum_{j=0}^{3}
\frac{H_j}{j!}
g_3^{(j)}
D^{\,3-j}.
\]


Differentiating the recursive identity with respect to the spatial variable
produces another recursion that will play a central role in the construction
of the compatibility matrices.

\begin{proposition}

The first spatial derivative satisfies
\[
DN_n
=
DN_{n-1}
+
\sum_{k=0}^{n-1}
\frac{H_k}{k!}
\left(
g_n^{(k+1)}
D^{\,n-k-1}
+
g_n^{(k)}
D^{\,n-k+1}
\right)
+
\frac{H_n}{n!}
g_n^{(n)}
D.
\]

More generally,
\[
D^mN_n
=
D^mN_{n-1}
+
D^m
\left(
\sum_{j=0}^{n}
\frac{H_j}{j!}
g_n^{(j)}
D^{\,n-j}
\right).
\]

\end{proposition}
\begin{proof}

Differentiating the recursive transport formula gives

\[
DN_n
=
DN_{n-1}
+
\sum_{j=0}^{n}
D
\left(
\frac{H_j}{j!}
g_n^{(j)}
D^{\,n-j}
\right).
\]
Since the coefficients $g_n^{(j)}(t)$ depend only on $t$, the derivative
acts only on the heat polynomial and the differential operator. Using the
Appell identity
\[
DH_j=jH_{j-1},
\]
one obtains
\[
D
\left(
\frac{H_j}{j!}
g_n^{(j)}
D^{\,n-j}
\right)
=
\frac{H_{j-1}}{(j-1)!}
g_n^{(j)}
D^{\,n-j}
+
\frac{H_j}{j!}
g_n^{(j)}
D^{\,n-j+1}.
\]

Reindexing the first sum with $k=j-1$ yields
\[
\sum_{k=0}^{n-1}
\frac{H_k}{k!}
g_n^{(k+1)}
D^{\,n-k-1},
\]
while the second sum remains unchanged. Separating the last term
$j=n$ produces
\[
\frac{H_n}{n!}
g_n^{(n)}
D,
\]
which proves the first identity.

The higher-order formula follows immediately by repeated differentiation
and successive applications of Leibniz's rule.

\end{proof}
Repeated application of Leibniz's rule yields explicit expressions for all
higher spatial derivatives. These identities provide the transport rows of
the compatibility matrices constructed in the following section.

\section{The boundary hierarchy}\label{Sec_4}

The transported operators introduced in the previous section describe the
evolution of the solution in the interior of the domain. In many
moving-boundary problems, however, the quantities of interest are
determined by evaluating the solution and its derivatives along the
boundary itself. This situation naturally arises, for example, in the
study of first-passage problems for Brownian motion, where the moving
boundary determines the hitting-time distribution
\cite{hernandez,HerGuerra}.

To obtain a closed compatibility system, the transported equations must be
supplemented by a hierarchy of boundary identities.

Throughout this work we consider solutions of the one-dimensional heat
equation
\[
v_t=\frac12D^2v,
\]
defined on the moving domain
\[
x<f(t),
\]
and satisfying the homogeneous Dirichlet condition

\[
v(t,f(t))=0.
\]

Differentiating this identity with respect to time and using the heat
equation yields the first boundary differential operator.


\begin{definition}

Define the first boundary operator by

\[
P_1(D,t)
=
D^2
+
2f'(t)D.
\]
Then every sufficiently smooth solution satisfies
\[
P_1(D,t)v=0,
\]
along the moving boundary \(x=f(t)\).
\end{definition}


Successive differentiation generates an infinite hierarchy of purely
spatial differential operators.

\begin{theorem}[Boundary recursion]

For every integer \(n\ge1\),
\[
P_{n+1}(D,t)
=
\left(
\frac12D^2
+
f'(t)D
\right)
P_n(D,t)
+
\frac{\partial}{\partial t}P_n(D,t),
\]
where the time derivative acts only on the coefficients of \(P_n\).

\end{theorem}


\begin{proof}

Assume that

\[
P_n(D,t)v=0
\]
holds along the moving boundary.

Differentiating with respect to time gives
\[
\frac{d}{dt}
\bigl(
P_n(D,t)v
\bigr)
=
0.
\]
Since the boundary itself moves,
\[
\frac{d}{dt}
=
\frac{\partial}{\partial t}
+
f'(t)D.
\]
Furthermore,
\[
v_t
=
\frac12D^2v.
\]
Therefore,
\[
0
=
\left(
\frac12D^2
+
f'D
\right)
P_n(D,t)v
+
\frac{\partial}{\partial t}
P_n(D,t)v,
\]
which immediately yields
\[
P_{n+1}
=
\left(
\frac12D^2
+
f'D
\right)
P_n
+
\frac{\partial}{\partial t}P_n.
\]
\end{proof}


The first members of the hierarchy are

\[
P_1
=
D^2
+
2f'D,
\]

\[
P_2
=
D^4
+
4f'D^3
+
4(f')^2D^2
+
4f''D,
\]
and
\[
\begin{aligned}
P_3
={}&
\frac12D^6
+
3f'D^5
+
6(f')^2D^4
\\
&
+
\left(
6f''
+
4(f')^3
\right)D^3
+
12f'f''D^2
+
4f'''D.
\end{aligned}
\]


\begin{remark}

The boundary hierarchy depends exclusively on the moving boundary
\(f(t)\) and is completely independent of the transported operators
introduced in the previous section. This separation between transport and
boundary information will play a fundamental role in the recursive
construction of the compatibility matrices. Indeed, the boundary rows will
remain unchanged throughout the recursive decomposition, whereas the
transport rows will evolve with the transported operators.

\end{remark}

\section{Compatibility matrices}\label{Sec_5}

The transported operators introduced in Section~3 govern the evolution of
the solution in the interior of the moving domain, whereas the boundary
operators introduced in Section~4 describe the constraints imposed by the
moving boundary.

A compatibility condition is obtained by requiring that these two
descriptions be simultaneously satisfied along the boundary
\[
x=f(t).
\]

Since the transported operators and the boundary operators involve
successively higher spatial derivatives of the unknown function, it is
natural to collect these equations into a single linear system. The
coefficient matrix of this system will be called the
\emph{compatibility matrix}.

\begin{definition}

Let $N_n$ be the transported operator of order $n$, and let
$P_1,\ldots,P_n$ denote the corresponding boundary hierarchy.

The compatibility matrix associated with $N_n$ is obtained by alternating
transport equations and boundary equations according to

\[
M_n
=
\begin{pmatrix}
\gamma_f\!\left(D^{\,n-1}N_n\right)
\\
P_n
\\
\gamma_f\!\left(D^{\,n-2}N_n\right)
\\
P_{n-1}
\\
\vdots
\\
\gamma_f(N_n)
\\
P_1
\end{pmatrix}.
\]
Here
\[
\gamma_f
\]
denotes restriction of the coefficients to the moving boundary
\(x=f(t)\).

\end{definition}

\begin{remark}

The transported operator $N_n$ has order $n$.

Consequently,
\[
D^{\,k}N_n,
\qquad
0\le k\le n-1,
\]
contains derivatives of the unknown function up to order
\[
2n-k.
\]
Similarly, the boundary operator $P_j$ has order $2j$.

The alternating construction above therefore produces a square linear
system whose unknowns are precisely the boundary derivatives of the
solution.

In particular,
\[
M_n
\]
has dimension
\[
2n\times2n.
\]
\end{remark}

The first compatibility matrix is
\[
M_1
=
\begin{pmatrix}
\gamma_f(N_2)
\\
P_1
\end{pmatrix},
\]
which is a \(2\times2\) matrix.

Explicitly,

\[
M_1
=
\begin{pmatrix}
g_2(t)
&
g_1(t)+(g_2)'(t)f(t)
\\
1
&
2f'
\end{pmatrix}.
\]

The second compatibility matrix is

\[
M_2
=
\begin{pmatrix}
\gamma_f(DN_3)
\\
P_2
\\
\gamma_f(N_3)
\\
P_1
\end{pmatrix},
\]
which is a \(4\times4\) matrix.

\begin{remark}

The compatibility matrices naturally decompose into two distinct families
of rows.

The transport rows
\[
\gamma_f\!\left(D^kN_n\right),
\]
originate from the transported differential operators and therefore depend
on the coefficients of the transported family.

The boundary rows
\[
P_j,
\]
depend exclusively on the geometry of the moving boundary.

This separation between transport information and boundary information is
the fundamental structural property of the compatibility matrices. It will
allow the recursive decomposition developed in the next section.

\end{remark}

The alternating structure of the compatibility matrices immediately
suggests a recursive construction. Indeed, when passing from $M_{n-1}$ to
$M_n$, the boundary hierarchy remains unchanged, while only the transport
rows are modified by the newly introduced transported coefficients. This
observation forms the basis of the recursive decomposition developed in the
next section.

\section{Recursive decomposition}\label{Sec_6}

The compatibility matrices constructed in the previous section naturally
split into transport and boundary contributions. The objective of this
section is to exploit the recursive structure of the transported operators
to obtain a corresponding recursive decomposition of the compatibility
matrices.


\subsection{Transport and boundary matrices}

Recall that the compatibility matrix associated with the transported
operator $N_{n+1}$ is

\[
M_n=
\begin{pmatrix}
\gamma_f(D^nN_{n+1})\\
P_{n+1}\\
\gamma_f(D^{n-1}N_{n+1})\\
P_n\\
\vdots\\
\gamma_f(N_{n+1})\\
P_1
\end{pmatrix}.
\]

It is convenient to separate transport information from boundary
information.

\begin{definition}

Define

\[
T_n=
\begin{pmatrix}
\gamma_f(D^nN_{n+1})\\
0\\
\gamma_f(D^{n-1}N_{n+1})\\
0\\
\vdots\\
\gamma_f(N_{n+1})\\
0
\end{pmatrix},
\]
and

\[
C_n=
\begin{pmatrix}
0\\
P_{n+1}\\
0\\
P_n\\
\vdots\\
0\\
P_1
\end{pmatrix}.
\]

Then
\[
M_n=T_n+C_n.
\]

\end{definition}

Notice that the boundary hierarchy depends only on the moving boundary
$f(t)$, whereas all recursive changes occur inside the transport matrix.


\subsection{Recursive transport rows}

The recursive transport identity immediately induces a recursive
decomposition of every transport row.

\begin{definition}

For every $n\ge2$, define the transport increment

\[
R_n
=
N_n-DN_{n-1}.
\]
\end{definition}

Since

\[
N_n
=
DN_{n-1}
+
R_n,
\]
one immediately obtains the following result.

\begin{lemma}
For every integer $m\ge0$,
\[
D^mN_n
=
D^{m+1}N_{n-1}
+
D^mR_n.
\]

\end{lemma}

\begin{proof}

Simply differentiate

\[
N_n=DN_{n-1}+R_n
\]
$m$ times with respect to the spatial variable.
\end{proof}


\subsection{Reference transport matrices}

Repeated application of the previous lemma allows every transport row to be
reduced recursively until only derivatives of the first transported
operator remain.

This motivates the following definition.

\begin{definition}

Define the reference transport matrix

\[
T_n^{(0)}
=
\begin{pmatrix}
\gamma_f(D^nN_1)\\
0\\
\gamma_f(D^{n-1}N_1)\\
0\\
\vdots\\
\gamma_f(DN_1)\\
0
\end{pmatrix}.
\]

The corresponding reference compatibility matrix is
\[
M_n^{(0)}
=
T_n^{(0)}
+
C_n.
\]

\end{definition}


\subsection{Perturbation matrices}

The remaining terms generated by the recursive decomposition are collected
according to their transport level.

\begin{definition}

For every $k\ge1$, define

\[
\Delta_k
=
\begin{pmatrix}
\gamma_f(D^{k-1}R_{k+1})\\
0\\
\gamma_f(D^{k-2}R_{k+1})\\
0\\
\vdots\\
\gamma_f(R_{k+1})\\
0
\end{pmatrix},
\]

where

\[
R_{k+1}=N_{k+1}-DN_k.
\]

Thus every perturbation matrix contains only transport information
introduced at the $(k+1)$-st transport level.

\end{definition}

For example,

\[
\Delta_1
=
\begin{pmatrix}
\gamma_f(R_2)\\
0
\end{pmatrix}
=
\begin{pmatrix}
\gamma_f(N_2-DN_1)\\
0
\end{pmatrix},
\]

while

\[
\Delta_2
=
\begin{pmatrix}
\gamma_f(DR_3)\\
0\\
\gamma_f(R_3)\\
0
\end{pmatrix}
=
\begin{pmatrix}
\gamma_f(DN_3-D^2N_2)\\
0\\
\gamma_f(N_3-DN_2)\\
0
\end{pmatrix}.
\]


\subsection{Recursive decomposition theorem}

The previous construction immediately yields the main structural result.

\begin{theorem}

For every integer $n\ge1$,
\[
T_n
=
T_n^{(0)}
+
\sum_{k=1}^{n}\Delta_k.
\]

Consequently,

\[
M_n
=
M_n^{(0)}
+
\sum_{k=1}^{n}\Delta_k.
\]

\end{theorem}

\begin{proof}

The proof proceeds by induction.

For $n=1$,

\[
N_2
=
DN_1
+
R_2,
\]
so that

\[
T_1
=
T_1^{(0)}
+
\Delta_1.
\]

Assume the decomposition holds for $T_{n-1}$.

Applying Lemma~6.1 to every transport row gives

\[
D^mN_{n+1}
=
D^{m+1}N_n
+
D^mR_{n+1},
\]
which replaces each transport row by a previously existing row together
with a correction. Collecting all correction rows produces the perturbation
matrix $\Delta_n$, yielding

\[
T_n
=
T_n^{(0)}
+
\sum_{k=1}^{n}\Delta_k.
\]

Adding the unchanged boundary matrix $C_n$ gives

\[
M_n
=
M_n^{(0)}
+
\sum_{k=1}^{n}\Delta_k.
\]

\end{proof}


\begin{remark}

The recursive decomposition completely separates the universal part of the
compatibility matrix from the information introduced at each transport
level.

The reference matrix

\[
M_n^{(0)}
\]
depends only on the first transported operator and on the boundary
hierarchy.

Each perturbation matrix

\[
\Delta_k
\]
contains exclusively the new transport information generated at level
$k+1$.

This decomposition will allow the determinant of the compatibility matrix
to be studied as a sequence of successive perturbations.

\end{remark}

\section{Low-rank perturbations}\label{Sec_7}

The recursive decomposition obtained in the previous section expresses the
compatibility matrix as

\[
M_n
=
M_n^{(0)}
+
\sum_{k=1}^{n}\Delta_k.
\]

The usefulness of this representation depends on the algebraic structure
of the perturbation matrices. The purpose of this section is to show that
each $\Delta_k$ is a low-rank perturbation of the reference matrix.


\subsection{Rank of the perturbation matrices}

Observe that every perturbation matrix contains only transport rows, while
the boundary rows vanish identically.

Indeed,

\[
\Delta_k
=
\begin{pmatrix}
\gamma_f(D^{k-1}R_{k+1})\\
0\\
\gamma_f(D^{k-2}R_{k+1})\\
0\\
\vdots\\
\gamma_f(R_{k+1})\\
0
\end{pmatrix},
\]

where

\[
R_{k+1}
=
N_{k+1}-DN_k.
\]

Consequently, exactly one half of its rows are identically zero.


\begin{proposition}

For every $k\ge1$,

\[
\operatorname{rank}(\Delta_k)\le k.
\]

\end{proposition}

\begin{proof}

The matrix $\Delta_k$ has $2k$ rows.

Among them, exactly $k$ rows vanish identically since the perturbation
affects only the transport equations.

Therefore the row space of $\Delta_k$ is generated by at most $k$
nonzero rows, implying

\[
\operatorname{rank}(\Delta_k)
\le
k.
\]

\end{proof}


\begin{remark}

The estimate above is purely structural and is independent of the
particular coefficients appearing in the transported operators.

In practice, the rank is frequently much smaller because the nonzero
transport rows often satisfy additional linear relations.

\end{remark}


\subsection{The Matrix Determinant Lemma}

Since every perturbation matrix has comparatively small rank, the
determinant of the compatibility matrix may be computed through successive
applications of the Matrix Determinant Lemma.

Recall the classical identity.

\begin{theorem}[Matrix Determinant Lemma]

Let $A$ be an invertible matrix and let

\[
U,V
\]
be matrices of compatible dimensions.

Then
\[
\det(A+UV^T)
=
\det(A)
\det\!\left(
I+V^TA^{-1}U
\right).
\]
\end{theorem}

Whenever
\[
\Delta_k
=
U_kV_k^T
\]
is a rank-$r_k$ factorization of the perturbation,
\[
r_k
=
\operatorname{rank}(\Delta_k),
\]
one obtains
\[
\det(M_n)
=
\det(M_n^{(0)})
\prod_{k=1}^{n}
\det
\!\left(
I
+
V_k^TM_{k-1}^{-1}U_k
\right),
\]
provided the intermediate matrices are invertible.


\subsection{Example: the matrix $M_1$}

For the first compatibility matrix,

\[
M_1
=
M_1^{(0)}
+
\Delta_1,
\]
where

\[
M_1^{(0)}
=
\begin{pmatrix}
g_1(t)&0\\
1&2f'
\end{pmatrix},
\]
and

\[
\Delta_1
=
\begin{pmatrix}
g_2(t)-g_1(t)&
g_1(t)+g_2'(t)f(t)\\
0&0
\end{pmatrix}.
\]

Since $\Delta_1$ has rank one, it may be written as

\[
\Delta_1
=
uv^T,
\]
for suitable vectors $u$ and $v$.

The Matrix Determinant Lemma therefore gives

\[
\det(M_1)
=
\det(M_1^{(0)})
\left(
1+
v^T(M_1^{(0)})^{-1}u
\right).
\]

Because

\[
\det(M_1^{(0)})
=
2f'g_1(t),
\]
the determinant naturally separates into a universal factor depending only
on the first transported operator and a correction factor generated by the
transport increment.


\subsection{Iterated perturbations}

The recursive decomposition obtained in Section~6 suggests viewing the
compatibility matrices as successive perturbations,

\[
M_n^{(0)}
\longrightarrow
M_n^{(0)}+\Delta_1
\longrightarrow
M_n^{(0)}+\Delta_1+\Delta_2
\longrightarrow
\cdots
\longrightarrow
M_n.
\]

Thus, rather than viewing the compatibility matrix as a single large
object, it may be regarded as an initial matrix together with a sequence
of increasingly higher-order transport corrections.

Whether this recursive viewpoint leads to an efficient computational
procedure remains an open question. Nevertheless, it provides a natural
hierarchical representation of the compatibility matrices and clarifies
the role played by each transported operator in the construction.

\section{Examples}\label{Sec_8}

\subsection{Example 1: The first compatibility matrix}

We begin with the simplest nontrivial example. Although the determinant of
the first compatibility matrix can be computed directly, this example
illustrates the recursive decomposition introduced in the previous
sections.


Consider the first transported operator

\[
N_1
=
g_1(t)D
+
(c_{0,0}+c_{0,1}x),
\]
where
\[
g_1(t)=c_{1,0}+c_{0,1}t.
\]

The second transported operator is

\[
N_2
=
g_2(t)D^2
+
\left(
g_1(t)
+
g_2'(t)x
\right)D
+
c_{0,0}+c_{0,1}x+c_{0,2}(x^2+t),
\]
with
\[
g_2(t)
=
c_{2,0}+c_{1,1}t+c_{0,2}t^2.
\]

Differentiating $N_1$ with respect to the spatial variable gives
\[
DN_1
=
g_1(t)D^2
+
c_{0,1}.
\]


Restricting the operators to the moving boundary yields

\[
\gamma_f(N_2)
=
g_2D^2
+
\left(
g_1+g_2'f
\right)D,
\]
and
\[
\gamma_f(DN_1)
=
g_1D^2.
\]
The boundary operator is
\[
P_1
=
D^2+2f'D.
\]
Therefore,
\[
M_1
=
\begin{pmatrix}
\gamma_f(N_2)
\\
P_1
\end{pmatrix}.
\]


Using the recursive decomposition,

\[
N_2
=
DN_1
+
R_2,
\]
where
\[
R_2
=
N_2-DN_1,
\]
the compatibility matrix splits as
\[
M_1
=
M_1^{(0)}
+
\Delta_1,
\]
with
\[
M_1^{(0)}
=
\begin{pmatrix}
g_1&0
\\
1&2f'
\end{pmatrix},
\]
and
\[
\Delta_1
=
\begin{pmatrix}
g_2-g_1&
g_1+g_2'f
\\
0&0
\end{pmatrix}.
\]
Notice that $\Delta_1$ affects only the transport row, while the boundary
row remains unchanged. This is a consequence of the fact that the boundary
hierarchy is completely independent of the transported operators.

Moreover,
\[
\det(M_1^{(0)})
=
2f'(t)\,g_1(t).
\]
Observe that the determinant of the reference matrix depends only on the
first transport coefficient $g_1(t)$. This phenomenon persists in the
higher-order examples considered below and motivates the study of the
reference determinants independently of the transport perturbations.

Since $\Delta_1$ has rank one, it admits a factorization
\[
\Delta_1
=
uv^T,
\]
for suitable vectors $u$ and $v$.

Consequently, the Matrix Determinant Lemma gives

\[
\det(M_1)
=
\det(M_1^{(0)})
\left(
1+
v^T(M_1^{(0)})^{-1}u
\right).
\]

Substituting the explicit expressions for $u$ and $v$ into the determinant
formula yields exactly the expression obtained by direct computation:
\[
\det(M_1)=2f'(t)g_2(t)-\left(g_1(t)+g_2'(t)f(t)\right).
\]


This example illustrates the main philosophy developed in this paper. The
compatibility matrix is first reduced to a universal reference matrix,
whose determinant depends only on the first transported operator, while
all higher-order information is incorporated through successive transport
perturbations.
\subsection{Example 2: The second compatibility matrix}

We now consider the second compatibility matrix. In contrast with the
previous example, the recursive transport decomposition now introduces two
successive perturbations.

The compatibility matrix is

\[
M_2=
\begin{pmatrix}
\gamma_f(DN_3)\\
P_2\\
\gamma_f(N_3)\\
P_1
\end{pmatrix}.
\]
The recursive transport formula gives
\[
N_3
=
DN_2
+
R_3,
\]
where
\[
R_3
=
N_3-DN_2.
\]
Applying the same recursion once more,
\[
N_2
=
DN_1
+
R_2,
\]
with
\[
R_2
=
N_2-DN_1.
\]
Substituting the second identity into the first yields
\[
N_3
=
D^2N_1
+
DR_2
+
R_3,
\]
and therefore
\[
DN_3
=
D^3N_1
+
D^2R_2
+
DR_3.
\]

Consequently,
\[
\gamma_f(DN_3)
=
\gamma_f(D^3N_1)
+
\gamma_f(D^2R_2)
+
\gamma_f(DR_3),
\]
and
\[
\gamma_f(N_3)
=
\gamma_f(D^2N_1)
+
\gamma_f(DR_2)
+
\gamma_f(R_3).
\]
Hence
\[
M_2
=
M_2^{(0)}
+
\Delta_2^{(1)}
+
\Delta_2^{(2)},
\]
where the reference matrix is
\[
M_2^{(0)}
=
\begin{pmatrix}
g_1 & 0 & 0 & 0\\
1 & 4f' & 4(f')^2 & 4f''\\
0 & g_1 & 0 & 0\\
0 & 0 & 1 & 2f'
\end{pmatrix},
\]
the first transport perturbation is
\[
\Delta_2^{(1)}
=
\begin{pmatrix}
g_2-g_1
&
g_1+g_2'f
&
2g_2'
&
0
\\
0&0&0&0
\\
0
&
g_2-g_1
&
g_1+g_2'f
&
g_2'
\\
0&0&0&0
\end{pmatrix},
\]
and the second transport perturbation is
\[
\Delta_2^{(2)}
=
\begin{pmatrix}
g_3-g_2
&
g_2+g_3'f
&
2g_3'
&
g_3''
\\
0&0&0&0
\\
0
&
g_3-g_2
&
g_2+g_3'f
&
g_3'
\\
0&0&0&0
\end{pmatrix}.
\]

Observe that each perturbation modifies only the transport rows, while the
boundary hierarchy remains completely unchanged.

The determinant of the reference matrix is

\[
\det(M_2^{(0)})
=
8g_1(t)^2
\left(
2(f'(t))^3-f''(t)
\right).
\]

Since both perturbation matrices satisfy
\[
\operatorname{rank}\!\left(\Delta_2^{(1)}\right)
=
\operatorname{rank}\!\left(\Delta_2^{(2)}\right)
=
2,
\]
each admits a factorization
\[
\Delta_2^{(i)}
=
U_iV_i^T,
\qquad
U_i,V_i\in\mathbb{R}^{4\times2}.
\]

The Matrix Determinant Lemma therefore yields
\[
\det(M_2^{(1)})
=
\det(M_2^{(0)})
\,
\det
\!\left(
I_2
+
V_1^T(M_2^{(0)})^{-1}U_1
\right),
\]
and
\[
\det(M_2)
=
\det(M_2^{(1)})
\,
\det
\!\left(
I_2
+
V_2^T(M_2^{(1)})^{-1}U_2
\right).
\]

Hence,
\[
\det(M_2)
=
\det(M_2^{(0)})
\,
\det
\!\left(
I_2
+
V_1^T(M_2^{(0)})^{-1}U_1
\right)
\,
\det
\!\left(
I_2
+
V_2^T(M_2^{(1)})^{-1}U_2
\right).
\]

The decomposition therefore replaces the computation of the determinant of
a single large compatibility matrix by a sequence of determinant updates
associated with low-rank perturbations. Whether this representation
ultimately leads to a computational advantage remains an open question.
Nevertheless, it provides a recursive description of the determinant that
reflects the recursive construction of the transported operators
themselves.

\section{Discussion and future directions}\label{Sec_9}

The main objective of this work has been to develop a recursive framework
for constructing the compatibility matrices associated with moving-boundary
problems. Rather than viewing these matrices as isolated algebraic objects,
the proposed approach exploits the recursive structure of the transported
operators to build them systematically.

The resulting decomposition

\[
M_n
=
M_n^{(0)}
+
\sum_{k=1}^{n}\Delta_k
\]
separates the compatibility matrix into two distinct components. The
reference matrix $M_n^{(0)}$ contains only the information associated with
the first transported operator together with the complete boundary
hierarchy, while each perturbation matrix $\Delta_k$ contains exclusively
the new transport information introduced at the $(k+1)$-st recursive
level. Consequently, every recursive update affects only the transport
rows, leaving the boundary hierarchy unchanged.

Although the present work does not provide a closed expression for
\[
\det(M_n),
\]
the recursive decomposition offers a different perspective on the
problem. Instead of computing the determinant of a large compatibility
matrix directly, one may regard it as the determinant of a reference
matrix together with a sequence of structured perturbations. This
observation naturally suggests the use of low-rank perturbation techniques,
such as the Matrix Determinant Lemma, to investigate the determinant
recursively.

The examples presented in this paper indicate that the reference matrices
possess additional algebraic structure. In particular, the first examples
suggest the existence of differential polynomials $Q_n(f)$ satisfying
\[
\det(M_n^{(0)})
=
c_n\,g_1(t)^n\,Q_n(f),
\]
for suitable constants $c_n$. At present this remains an empirical
observation, but it suggests that the reference determinants may admit a
much simpler description than the compatibility matrices themselves.

Likewise, the perturbation matrices appear to satisfy considerably
stronger structural properties than those established here. While the
present analysis proves only general rank estimates, numerical
experiments suggest that additional linear dependencies may exist among
their transport rows. A better understanding of these perturbations could
lead to substantially simpler recursive determinant formulas.

Several natural directions for future investigation emerge from this work.

\begin{itemize}

\item
Obtain a closed recurrence for the differential polynomials
$Q_n(f)$ appearing in the reference determinants.

\item
Determine canonical low-rank factorizations for the perturbation matrices
$\Delta_k$.

\item
Develop recursive formulas for $\det(M_n)$ based on successive
applications of the Matrix Determinant Lemma.

\item
Extend the recursive transport construction to more general parabolic
operators and to multidimensional moving-boundary problems.

\item
Investigate possible connections between the recursive compatibility
matrices developed here and other recursive constructions appearing in
integrable systems, inverse problems, and free-boundary theory.

\end{itemize}

More generally, the recursive point of view developed in this paper
suggests that compatibility matrices possess a hierarchical algebraic
structure that has not been fully exploited. Whether this perspective
ultimately leads to explicit formulas for the moving boundary remains an
open question. Nevertheless, we believe that the recursive decomposition
introduced here provides a useful framework for studying these problems and
may serve as the starting point for a broader algebraic theory of
compatibility matrices associated with moving boundaries.

\nocite{*}

\bibliographystyle{plain}
\bibliography{references}

\end{document}